\documentclass[12pt,reqno]{amsart}
\usepackage{amsmath}
\usepackage{amssymb}
\usepackage[left=3cm,top=3cm,right=3cm,bottom=3cm]{geometry}
\usepackage{epsfig}

\begin{document}
\newtheorem{theorem}{Theorem}[section]
\newtheorem{lemma}[theorem]{Lemma}
\newtheorem{definition}[theorem]{Definition}
\newtheorem{conjecture}[theorem]{Conjecture}
\newtheorem{proposition}[theorem]{Proposition}
\newtheorem{claim}[theorem]{Claim}
\newtheorem{algorithm}[theorem]{Algorithm}
\newtheorem{corollary}[theorem]{Corollary}
\newtheorem{observation}[theorem]{Observation}
\newtheorem{problem}[theorem]{Open Problem}
\newcommand{\noin}{\noindent}
\newcommand{\ind}{\indent}
\newcommand{\om}{\omega}
\newcommand{\pp}{\mathcal P}
\newcommand{\AC}{\mathcal A \mathcal C}
\newcommand{\bAC}{\overline{\AC}}
\newcommand{\ppp}{\mathfrak P}
\newcommand{\N}{{\mathbb N}}
\newcommand{\LL}{\mathbb{L}}
\newcommand{\R}{{\mathbb R}}
\newcommand{\E}{\mathbb E}
\newcommand{\Prob}{\mathbb{P}}
\newcommand{\eps}{\varepsilon}

\newcommand{\Ss}{{\mathcal S}}
\newcommand{\Nn}{{\mathcal N}}

\newcommand{\ceil}[1]{\left \lceil #1 \right \rceil}
\newcommand{\floor}[1]{\left \lfloor #1 \right \rfloor}
\newcommand{\size}[1]{\left \vert #1 \right \vert}
\newcommand{\dist}{\mathrm{dist}}

\title{A note on the acquaintance time of random graphs}

\author{William B. Kinnersley}
\address{Department of Mathematics, Ryerson University, Toronto, ON, Canada, M5B 2K3}
\email{\tt wkinners@ryerson.ca}

\author{Dieter Mitsche}
\address{Universit\'{e} de Nice Sophia-Antipolis, Laboratoire J.A. Dieudonn\'{e}, Parc Valrose, 06108 Nice cedex 02}
\email{\texttt{dmitsche@unice.fr}}

\author{Pawe\l{} Pra\l{}at}
\address{Department of Mathematics, Ryerson University, Toronto, ON, Canada}
\email{\tt pralat@ryerson.ca}

\thanks{The authors gratefully acknowledge support from NSERC and Ryerson University}

\begin{abstract}
In this short note, we prove a conjecture of Benjamini, Shinkar, and Tsur on the acquaintance time $\AC(G)$ of a random graph $G \in G(n,p)$. It is shown that asymptotically almost surely $\AC(G) = O(\log n / p)$ for $G \in G(n,p)$, provided that $pn - \log n - \log \log n \to \infty$ (that is, above the threshold for Hamiltonicity). Moreover, we show a matching lower bound for dense random graphs, which also implies that asymptotically almost surely $K_n$ cannot be covered with $o(\log n / p)$ copies of a random graph $G \in G(n,p)$, provided that $pn > n^{1/2+\eps}$ and $p < 1-\eps$ for some $\eps>0$. We conclude the paper with a small improvement on the general upper bound showing that for any $n$-vertex graph $G$, we have $\AC(G) = O(n^2/\log n )$.
\end{abstract}

\maketitle

\section{Introduction}\label{sec:intro}

In this paper, we study the following graph process, which was recently introduced by Benjamini, Shinkar, and Tsur~\cite{bst}. Let $G=(V,E)$ be a finite connected graph. We start the process by placing one \emph{agent} on each vertex of $G$. Every pair of agents sharing an edge is declared to be \emph{acquainted}, and remains so throughout the process. In each round of the process, we choose some matching $M$ in $G$.  ($M$ need not be maximal; perhaps it is a single edge.) For each edge of $M$, we swap the agents occupying its endpoints, which may cause more agents to become acquainted. The \emph{acquaintance time} of $G$, denoted by $\AC(G)$, is the minimum number of rounds required for all agents to become acquainted with one another.

It is clear that 
\begin{equation}\label{eq:trivial_lower}
\AC(G) \ge \frac {{|V| \choose 2}}{|E|} - 1, 
\end{equation}
since $|E|$ pairs are acquainted initially, and at most $|E|$ new pairs become acquainted in each round. In~\cite{bst}, it was shown that always $\AC(G) = O(\frac{n^2}{\log n / \log \log n})$, where $n = |V|$.  Moreover, for all functions $f : \N \rightarrow \N$ with $1 \le f(n) \le n^{1.5}$, the authors constructed families $\{G_n\}$ of graphs with $\size{V(G_n)} = n$ for all $n$ such that $\AC(G_n) = \Theta(f_n)$. The problem is similar in flavour to the problems of Routing Permutations on Graphs via Matchings~\cite{ACG94}, Gossiping and Broadcasting~\cite{HHL88}, and Target Set Selection~\cite{KKT03, Che09, Rei12}.

\bigskip

In this paper, we consider the acquaintance time of binomial random graphs. The \emph{random graph} $G(n,p)$ consists of the probability space $(\Omega, \mathcal{F}, \Prob)$, where $\Omega$ is the set of all graphs with vertex set $[n]=\{1,2,\dots,n\}$, $\mathcal{F}$ is the family of all subsets of $\Omega$, and for every $G \in \Omega$,
$$
\Prob(G) = p^{|E(G)|} (1-p)^{{n \choose 2} - |E(G)|} \,.
$$
This space may be viewed as the set of outcomes of ${n \choose 2}$ independent coin flips, one for each pair $(u,v)$ of vertices, where the probability of success (that is, adding edge $uv$) is $p.$ Note that $p=p(n)$ may tend to zero as $n$ tends to infinity. All asymptotics throughout are as $n \rightarrow \infty $ (we emphasize that the notations $o(\cdot)$ and $O(\cdot)$ refer to functions of $n$, not necessarily positive, whose growth is bounded). We say that an event in a probability space holds \emph{asymptotically almost surely} (or \emph{a.a.s.}) if the probability that it holds tends to $1$ as $n$ goes to infinity. 

For constant $p$, observe that $\log_{\frac{1}{1-p}}n=\Theta(\log n)$, but for $p=o(1)$ we have
$$
\log_{\frac{1}{1-p}}n = \frac {\log n}{-\log (1-p)} = (1+o(1)) \frac {\log n}{p}.
$$

\bigskip

Let $G \in G(n,p)$ with $p = p(n) \ge (1+\eps) \log n / n$ for some $\eps > 0$. (Recall that $\AC(G)$ is defined only for connected graphs, and $\log n / n$ is the threshold for connectivity in $G(n,p)$---see, for example,~\cite{bol, JLR} for more.) Since a.a.s.\ $|E(G)| = (1+o(1)) {n \choose 2} p$, it follows immediately from the trivial lower bound~(\ref{eq:trivial_lower}) that a.a.s.\ $\AC(G) = \Omega(1/p)$. On the other hand, it is known that a.a.s.\ $G$ has a Hamiltonian path, which implies that a.a.s.\ $\AC(G) = O(n)$ (see~\cite{bst} or Lemma~\ref{lem:hamilton} below). Despite the fact that no non-trivial upper bound on $\AC(G)$ was known, it was conjectured in~\cite{bst} that a.a.s.\ $\AC(G) = O(\text{poly} \log(n) / p)$. We confirm this conjecture.

\begin{theorem}\label{thm:upper_gnp}
Let $\eps > 0$ and $(1+\eps) \log n / n \le p=p(n) \le 1-\eps$. For $G \in G(n,p)$, a.a.s.
$$
\AC(G) = O \left( \frac {\log n}{p} \right).
$$
\end{theorem}

 In fact, note the following: whenever $G_2$ is a subgraph of $G_1$ on the same vertex set,  $\AC(G_1) \le \AC(G_2)$, since the agents in $G_1$ have more edges to use. Hence, for any $p \ge 0.99$ (possibly $p \to 1$) and $G_1 \in G(n,p)$, we have that $\AC(G_1) \le \AC(G_2)$, where $G_2 \in G(n,0.99)$. Since a.a.s.\  $\AC(G_2) = O(\log n)$, a.a.s.\ $\AC(G_1) = O(\log n)$, and so the condition $p < 1-\eps$ in the theorem can be eliminated. Clearly, for denser graphs, this upper bound might not be tight; in particular, for the extreme case $p=1$, we trivially have $\AC(G_2)=\AC(K_n)=0$. Moreover, since the threshold for Hamiltonicity in $G(n,p)$ is $p=(\log n+\log \log n)/n$ (see, for example,~\cite{bol}), and for a Hamiltonian graph we have $\AC(G)=O(n)$, it follows that a.a.s.\ $\AC(G) = O(n)$, provided that $pn - \log n -\log \log n \to \infty$. So the desired bound for the acquaintance time holds at the time a random graph becomes Hamiltonian. We get the following corollary. 

\begin{corollary}
Suppose that $p=p(n)$ is such that $pn - \log n -\log \log n \to \infty$. For $G \in G(n,p)$, a.a.s.
$$
\AC(G) = O \left( \frac {\log n}{p} \right).
$$
\end{corollary}

\bigskip

In hopes of improving the trivial lower bound on the acquaintance time of $G(n,p)$, we consider a variant of the original process.  Suppose that each agent has a helicopter and can, on each round, move to any vertex she wants. (We retain the requirement that no two agents can occupy a single vertex simultaneously.)  In other words, in every step of the process, the agents choose some permutation $\pi$ of the vertices, and the agent occupying vertex $v$ flies directly to vertex $\pi(v)$, regardless of whether there is an edge or even a path between $v$ and $\pi(v)$. (In fact, it is no longer necessary that the graph be connected.) Let $\bAC(G)$ be the counterpart of $\AC(G)$ under this new model, that is, the minimum number of rounds required for all agents to become acquainted with one another. Since helicopters make it easier for agents to get acquainted, we immediately get that for every graph $G$, 
\begin{equation}\label{eq:bACvsAC}
\bAC(G) \le \AC(G).
\end{equation}
On the other hand, $\bAC(G)$ also represents the minimum number of copies of a graph $G$ needed to cover all edges of a complete graph of the same order. Thus inequality~(\ref{eq:trivial_lower}) can be strengthened to $\bAC(G) \ge {|V| \choose 2}/|E| - 1$.

We prove the following lower bound on $\bAC(G)$ (and hence on $\AC(G)$).  This result also implies that a.a.s.\ $K_n$ cannot be covered with $o(\log n / p)$ copies of a dense random graph $G \in G(n,p)$.

\begin{theorem}\label{thm:lower_gnp}
Let $\eps > 0$, $p=p(n) \ge n^{-1/2+\eps}$  and $p \le 1-\eps$. For $G \in G(n,p)$, a.a.s.
$$
\AC(G) \ge \bAC(G) \ge \frac {\eps}{2} \log_{1/(1-p)} n =  \Omega \left( \frac {\log n}{p} \right).
$$
\end{theorem}

Theorem~\ref{thm:upper_gnp} and Theorem~\ref{thm:lower_gnp} together determine the order of growth for the acquaintance time of dense random graphs (in particular, random graphs with average degree at least $n^{1/2+\eps}$ for some $\eps > 0$). 

\begin{corollary}
Let $\eps > 0$, $p=p(n) \ge n^{-1/2+\eps}$  and $p \le 1-\eps$. For $G \in G(n,p)$, a.a.s.
$$
\bAC(G) = \Theta \Big( \AC(G) \Big) = \Theta \left( \frac {\log n}{p} \right).
$$
\end{corollary}

The behaviours of $\AC(G)$ and $\bAC(G)$ for sparser random graphs remain undetermined.

We conclude the paper with a small improvement on the general upper bound, showing that for every $n$-vertex graph we have $\AC(G) = O(n^2/ \log n)$, a bound that is smaller than the previously known upper bound by a multiplicative factor of $\log \log n$.

\section{Proofs}

\subsection{Proof of Theorem~\ref{thm:upper_gnp}}

We will use the fact, observed in~\cite{bst}, that for any graph $G$ on $n$ vertices with a Hamiltonian path, we have $\AC(G)=O(n)$. We need a slightly stronger statement, so we provide a different argument.

\begin{lemma}\label{lem:hamilton}
Let $G$ be a graph on $n$ vertices.  If $G$ has a Hamiltonian path, then there exists a strategy ensuring that within $2n$ rounds every pair of agents gets acquainted (in particular, $\AC(G)=O(n)$) and, moreover, that every agent visits every vertex.
\end{lemma}
\begin{proof}
Index the vertices of $G$ as $v_1, v_2, \ldots, v_n$ so that $P=(v_1, v_2, \ldots, v_n)$ is a Hamiltonian path. For $1 \le i \le n-1$, let $e_i = v_iv_{i+1}$. On odd-numbered rounds, swap agents on all odd-indexed edges; on even-numbered rounds, swap agents on all even-indexed edges.  This has the following effect.  Agents that begin on odd-indexed vertices move ``forward'' in the vertex ordering, pause for one round at $v_n$, move ``backward'', pause again at $v_1$, and repeat; agents that begin on even-indexed vertices move backward, pause at $v_1$, move forward, pause at $v_n$, and repeat.  After $2n$ rounds, each agent has traversed the entire path; in doing so, she has necessarily passed by every other agent.
\end{proof}

We are now ready to prove Theorem~\ref{thm:upper_gnp}. 
\begin{proof}[Proof of Theorem~\ref{thm:upper_gnp}]
In order to avoid technical problems with events not being independent, we use a classic technique known as \emph{two-round exposure}. The observation is that a random graph $G \in G(n,p)$ can be viewed as a union of two independently generated random graphs $G_1 \in G(n,p_1)$ and $G_2 \in G(n,p_2)$, provided that $p=p_1 + p_2 - p_1 p_2$ (see, for example,~\cite{bol, JLR} for more information). 

Let $p_1 := (1+\eps/2) \log n / n$ and
$$
p_2 := \frac {p-p_1}{1-p_1} \ge p-p_1 \ge \frac {\eps/2}{1+\eps} p
$$
(recall that $p \ge (1+\eps) \log n / n$). Fix $G_1 \in G(n,p_1)$ and $G_2 \in G(n,p_2)$, with $V(G_1) = V(G_2) = \{v_1, v_2, \ldots, v_n\}$, and view $G$ as the union of $G_1$ and $G_2$.  It is known that a.a.s.\ $G_1$ has a Hamiltonian path (as usual, see~\cite{bol, JLR} for more). Hence we may suppose that $P=(v_1, v_2, \ldots, v_n)$ is a Hamiltonian path of $G_1$ (and thus also of $G$).

Now let $k = k(n) = 2.5 \log_{1/(1-p_2)} n$. We partition the path $P$ into many paths, each on $k$ vertices. This partition also divides the agents into $\ceil{n/k}$ teams, each team consisting of $k$ agents (except for the ``last'' team, which may be smaller). Every team performs (independently and simultaneously) the strategy from Lemma~\ref{lem:hamilton}. It follows that the length of the full process is at most $2k = 5 \log_{1/(1-p_2)} n$, which is asymptotic to
$$
5 \frac {\log n}{p_2} \le 10 \frac {(1+\eps)}{\eps} \frac {\log n}{p} = O\left( \frac {\log n}{p} \right)
$$
provided that $p = o(1)$; if instead $p = \Omega(1)$, then the number of rounds needed is clearly $O(\log n)$. Moreover, every pair of agents from the same team gets acquainted.

It remains to show that a.a.s.\ every pair of agents from different teams gets acquainted. Let us focus on one such pair. It follows from Lemma~\ref{lem:hamilton} that each agent, excepting those in the ``last'' team, visits $k$ distinct vertices.  Since the agents belong to different teams, at least one belongs to a team of size $k$, so the two agents occupy at least $k$ distinct pairs of vertices during the process. Considering only those edges in $G_2$, the probability that the two agents never got acquainted is at most
$$
(1-p_2)^{k} = o(n^{-2}).
$$
Since there are at most $n \choose 2$ pairs of agents, the result holds by the union bound.
\end{proof}

We now turn to Theorem~\ref{thm:lower_gnp} and the helicopter variant of the acquaintance process.

\subsection{Proof of Theorem~\ref{thm:lower_gnp}}

The first inequality in the statement of the theorem is~\eqref{eq:bACvsAC}. It remains to show the desired lower bound for $\bAC(G)$.

\begin{proof}[Proof of Theorem~\ref{thm:lower_gnp}]
Let $a_1, a_2, \ldots, a_n$ denote the $n$ agents, and let $A = \{a_1, a_2, \ldots, a_n\}$. Take $k = \frac {\eps}{2} \log_{1/(1-p)} n$ and fix $k$ bijections $\pi_i : A \to V(G)$, for $i \in \{0, 1, \ldots k-1\}$.  This corresponds to fixing a $(k-1)$-round strategy for the agents; in particular, agent $a_j$ occupies vertex $\pi_i(a_j)$ in round $i$. We aim to show that at the end of the process (that is, after $k-1$ rounds) the probability that all agents are acquainted is only $o((1/n!)^k)$. This completes the proof: the number of choices for $\pi_0, \pi_1, \ldots, \pi_{k-1}$ is $(n!)^k$, so by the union bound, a.a.s.\ no strategy makes all pairs of agents acquainted.

To estimate the probability in question, we consider the following analysis, which iteratively exposes edges of a random graph $G \in G(n,p)$. For any pair $r = \{a_x, a_y\}$ of agents, we consider all pairs of vertices visited by this pair of agents throughout the process:
$$
S(r) = \{ \{\pi_i(a_x), \pi_i(a_y)\} : i \in \{0, 1, \ldots k-1\} \}.
$$
Clearly, $1 \le |S(r)| \le k$. Take any pair $r_1$ of agents and expose the edges of $G$ in $S(r_1)$, one by one.  If we expose all of $S(r_1)$ without discovering an edge, then we discard $r_1$ and proceed.  (In fact we have just learned that the pair $r_1$ never gets acquainted, so we could choose to halt the process immediately.  However, to simplify the analysis, we continue normally.)  If instead we do discover some edge $e$ of $G$, then we discard all pairs of agents that ever occupy this edge (that is, we discard all pairs $r$ such that $e \in S(r)$).  In either case, we shift our attention to another pair $r_2$ of agents (chosen arbitrarily from among the pairs not yet discarded). It may happen that some of the pairs of vertices in $S(r_2)$ have already been exposed, but the analysis guarantees that no edge has yet been discovered. Let $T(r_2) \subseteq S(r_2)$ be the set of edges in $S(r_2)$ not yet exposed. As before, we expose these edges one by one, until either we discover an edge or we run out of edges to expose. If an edge is discovered, then we again discard all pairs that ever occupy that edge. 

We continue this process until all available pairs of agents have been investigated. Since one pair of agents can force us to discard at most $k$ pairs (including the original pair), the process investigates at least ${n \choose 2} / k$ pairs of agents. Moreover, among these pairs, the probability $\Prob(r_t)$ that pair $r_t$ gets acquainted is
$$
\Prob(r_t) = 1-(1-p)^{|T(r_t)|} \le 1-(1-p)^{|S(r_t)|} \le 1-(1-p)^{k}.
$$
Hence, the probability that all pairs get acquainted is at most
\begin{eqnarray*}
\prod_{t=1}^{{n \choose 2} / k} \Prob(r_t) &\le& \left( 1-(1-p)^{k} \right)^{{n \choose 2} / k} \le \exp \left( -(1-p)^{k} {n \choose 2} / k\right) \\
&\le& \exp \left( -n^{-\eps/2} {n \choose 2} 3 n^{-1/2+\eps/2} \right) \le \exp \left( -n^{3/2}  \right)  \\
&\le& \exp \left( -n^{1+\eps/2} k \right) = o \left(\exp \left( - k n \log n \right) \right) \\
&=& o \left( \left(1/n! \right)^k \right),
\end{eqnarray*}
since $k = \Theta(\log n / p) \le n^{1/2-\eps/2}/3$.
As mentioned earlier, it follows that a.a.s.\ $\bAC(G) \ge k$, and the proof is finished.
\end{proof}

\section{General Upper Bound}

We conclude this note with a small improvement to another result of Benjamini, Shinkar, and Tsur.  They proved the following (\cite{bst}, Theorem 5.5):
\begin{theorem}[\cite{bst}]\label{thm:bst_upper}
For every $n$-vertex graph $G$, we have that $\AC(G) = O\left (\frac{n^2}{\log n / \log \log n} \right )$.
\end{theorem}   
However, they ask whether in fact this bound can be improved to $O(n^{1.5})$.  While we are unable to resolve this question, we do provide the following minor improvement:
\begin{theorem}\label{thm:upper_improved}
For every $n$-vertex graph $G$, we have that $\AC(G) = O\left (\frac{n^2}{\log n} \right )$.
\end{theorem}

We use the following result from~\cite{bst} (Claim 2.1 in that paper).
\begin{claim}[\cite{bst}]\label{claim:bst_route}
Let $G = (V,E)$ be a tree. Let $S, T \subseteq V$ be two subsets of the vertices of equal
size $k = |S| = |T|$, and let $\ell = \max_{v\in S,u\in T} {\dist(v, u)}$ be the maximal distance between a
vertex in $S$ and a vertex in $T$. Then, there is a strategy of $\ell + 2(k - 1)$ matchings that
routes all agents from $S$ to $T$.
\end{claim}

Before proving the main result, we present two simple propositions.  A {\em caterpillar} is a tree in which all vertices are either on or adjacent to a single path (known as the {\em spine} of the caterpillar).

\begin{proposition}\label{prop:caterpillar}
If $T$ is an $n$-vertex caterpillar, then $\AC(T) = O(n)$.
\end{proposition}
\begin{proof}
Let $k$ denote the number of vertices in the spine of $T$, and note that the diameter of $T$ is $O(k)$.  Partition the $n$ agents into $\ceil{n/k}$ teams, each of size at most $k$.  We iteratively route a team onto the spine, apply the strategy in Lemma~\ref{lem:hamilton}, and repeat until all teams have traversed the spine.  When a team traverses the spine, all team members meet all other agents in the graph, so this strategy suffices to acquaint all pairs of agents.  By Claim~\ref{claim:bst_route} and Lemma~\ref{lem:hamilton}, each iteration can be completed in $O(k)$ rounds, so the total number of rounds needed is $O(n)$.
\end{proof}

\begin{proposition}\label{prop:subcat_logn}
If $T$ is an $n$-vertex tree, then $T$ contains a caterpillar on at least $\log_2 n$ vertices.
\end{proposition}
\begin{proof}
We use induction on $n$ to prove the following stronger statement: for every vertex $r$, the tree $T$ contains a caterpillar on at least $\log_2 n$ vertices, in which $r$ is an endpoint of the spine.  When $n \le 2$ the claim is trivial, so suppose otherwise.  View $T$ as being rooted at $r$.  Let $d = \deg(r)$, and consider the $d$ subtrees rooted at the children of $r$.  The largest of these subtrees, say $T'$, contains at least $(n-1)/d$ vertices.  The induction hypothesis guarantees a caterpillar in $T'$ having at least $\log_2 [(n-1)/d]$ vertices, in which one endpoint of the spine is adjacent to $r$.  Appending $r$ and its other children to this caterpillar yields a caterpillar of the desired form in $T$ with at least $\log_2 [(n-1)/d] + d$ vertices; as $\log_2 [(n-1)/d] + d \ge \log_2 n$, this completes the proof.
\end{proof}

Theorem~\ref{thm:upper_improved} now follows easily, using the same approach as in~\cite{bst}, Theorem 5.5.

\begin{proof}[Proof of Theorem~\ref{thm:upper_improved}]
By monotonicity of $\AC(G)$, we may suppose $G$ is a tree.  Let $T$ be the largest caterpillar contained in $G$, and let $k = \size{V(T)}$.  (Note that $G$ has diameter $O(k)$.)  Partition the agents into $\ceil{2n/k}$ teams, each of size at most $k/2$.  For each pair of distinct teams, route both teams onto $T$ and apply the strategy in Proposition~\ref{prop:caterpillar}; this ensures that any two agents become acquainted.  There are $O(n^2/k^2)$ pairs of teams, and (by Claim~\ref{claim:bst_route} and Proposition~\ref{prop:caterpillar}) we spend $O(k)$ rounds for each pair, so the entire process lasts for $O(n^2/k)$ rounds.  By Proposition~\ref{prop:subcat_logn} we have $k \ge \log_2 n$, which completes the proof.
\end{proof}

\section{Addendum}

There have been several developments on this problem since the submission of this paper. In this note, we proved a conjecture from~\cite{bst} on the acquaintance time $\AC(G)$ of the random graph $G \in G(n,p)$. 
Moreover, we showed that for any $n$-vertex graph $G$, we have $\AC(G) = O(n^2/\log n )$. 
This general upper bound was recently improved in~\cite{AS}: they show that $\AC(G) = O(n^{3/2})$, which was conjectured in~\cite{bst} and is tight up to a multiplicative constant. 

In~\cite{MP}, the acquaintance time of a random subgraph of a random geometric graph $G \in G(n,r,p)$ is studied. (In $G \in G(n,r,p)$, $n$ vertices are chosen uniformly at random and independently from $[0,1]^2$, and two vertices are adjacent with probability $p$ if the Euclidean distance between them is at most $r$.)
Asymptotic results for the acquaintance time of $G \in G(n,r,p)$ for a wide range of $p=p(n)$ and $r=r(n)$ are presented. In particular, it is shown that with high probability $\AC(G) = \Theta(r^{-2})$ for $G \in G(n,r,1)$, the classic random geometric graph, provided that $\pi n r^2 - \ln n \to \infty$ (that is,  above the connectivity threshold). For the percolated random geometric graph $G \in G(n,r,p)$, it follows that with high probability $\AC( G ) = O(r^{-2} p^{-1} \ln n)$, provided that $\pi n r^2 p \geq K \ln n$ for some large constant $K>0$ and $p < 1-\eps$ for some $\eps>0$. Moreover, a matching lower bound for dense random percolated graphs is presented, which also implies that with high probability $K_n$ cannot be covered with $o(r^{-2} p^{-1} \ln n)$ copies of a random geometric graph $G \in G(n,r,p)$, provided that $pr \geq n^{-1/2+\eps}$ and $p < 1-\eps$ for some $\eps>0$.


\begin{thebibliography}{99}

\bibitem{ACG94} N.~Alon, F.R.K.~Chung, and R.L.~Graham. Routing permutations on graphs via matchings, \emph{SIAM J.\ Discrete Math.} \textbf{7} (1994), 513--530.

\bibitem{AS} O.\ Angel and I.\ Shinkar, A tight upper bound on acquaintance time of graphs, preprint.

\bibitem{bst} I.~Benjamini, I.~Shinkar, and G.~Tsur, Acquaintance Time of a Graph, preprint.

\bibitem{bol} B.\ Bollob\'{a}s, \emph{Random Graphs}, Cambridge University Press, Cambridge, 2001.

\bibitem{Che09} N.~Chen, On the approximability of influence in social networks, \emph{SIAM Journal on Discrete Mathematics} \textbf{23(5)} (2009), 1400--1415.

\bibitem{HHL88} S.T.~Hedetniemi, S.M.~Hedetniemi, and A.~Liestman. A survey of gossiping and broadcasting in communication networks, \emph{Networks} \textbf{18(4)} (1998), 319--349.

\bibitem{JLR} S.\ Janson, T.\ {\L}uczak, A.\ Ruci\'nski, \emph{Random Graphs}, Wiley, New York, 2000.

\bibitem{KKT03} D.~Kempe, J.~Kleinberg, and E.~Tardos, Maximizing the spread of influence through a social network, In KKD, 137--146, 2003.

\bibitem{MP} T.\ M\"{u}ller, P.\ Pra\l{}at, The acquaintance time of (percolated) random geometric graphs, preprint.

\bibitem{Rei12} D.~Reichman, New bounds for contagious sets, \emph{Discrete Mathematics} \textbf{312} (2012), 1812--1814.

\end{thebibliography}
\end{document}